\documentclass[reqno]{amsart}
\usepackage{amssymb}
\usepackage{nicefrac}
\usepackage[mathscr]{eucal}
\usepackage{hyperref}
\usepackage{graphicx}
\usepackage{xcolor}

\theoremstyle{plain}
\newtheorem{prop}{Proposition}[section]
\newtheorem{thm}[prop]{Theorem}
\newtheorem{cor}[prop]{Corollary}
\newtheorem{lem}[prop]{Lemma}

\theoremstyle{definition}
\newtheorem{dfn}[prop]{Definition}
\newtheorem{rem}[prop]{Remark}

\newtheorem{lab}[prop]{}

\newcommand{\N}{{\mathbb{N}}}

\newcommand{\R}{{\mathbb{R}}}
\newcommand{\Z}{{\mathbb{Z}}}

\newcommand{\sfS}{\mathsf{S}}

\newcommand{\SOCF}{\operatorname{SOCF}}
\newcommand{\lin}{\operatorname{lin}}

\DeclareMathOperator{\codim}{codim}

\DeclareMathOperator{\im}{im}

\DeclareMathOperator{\sxdeg}{sxdeg}

\DeclareMathOperator{\supp}{supp}

\DeclareTextFontCommand{\textnf}{\normalfont}

\newcommand{\CP}{\mathrm{CP}}

\newcommand{\conv}{\mathrm{conv}}

\newcommand{\ord}{{\rm ord}}

\renewcommand{\emptyset}{\varnothing}
\renewcommand{\setminus}{\smallsetminus}
\newcommand{\ol}{\overline}
\newcommand{\plus}{{\scriptscriptstyle+}}
\newcommand{\all}{\forall\,}
\newcommand{\ex}{\exists\,}

\renewcommand{\subset}{\subseteq}
\renewcommand{\supset}{\supseteq}

\newcommand{\sa}{semialgebraic}

\newcommand{\bil}[2]{\langle{#1},{#2}\rangle}
\newcommand{\sprod}[2]{\left< #1 \, , \, #2 \right>} 

\newcommand{\setcond}[2]{\left\{#1\,:\,#2\right\}}

\newcommand{\oldversion}[1]{}

\newcommand{\marginnote}[1]{\vrule width0pt height0pt depth0pt
    \vadjust{\vbox to0pt{\vss\hbox to\hsize{\hskip\hsize\quad
    #1\hss}\vskip1.5pt}}}

\begin{document}

\title{Convex hulls of monomial curves,\\and a sparse
positivstellensatz}

\author{Gennadiy Averkov}
\address{G. Averkov, Brandenburgische Technische Universit\"at Cottbus-Senftenberg, Platz der Deutschen Einheit 1, D-03046, Germany} 
\email{averkov@b-tu.de}
\author{Claus Scheiderer}
\address{C. Scheiderer, Fachbereich Mathematik und Statistik, Universit\"at Konstanz, D-78457 Konstanz, Germany}
\email{claus.scheiderer@uni-konstanz.de}

\date\today
\maketitle

\begin{abstract}
Consider  the closed convex hull  $K$ of a monomial curve given
parametrically as $(t^{m_1},\ldots,t^{m_n})$, with the parameter $t$
varying in an interval $I$. We show, using constructive arguments,
that $K$ admits a lifted semidefinite description by $\mathcal{O}(d)$
linear matrix inequalities (LMIs), each of size
$\left\lfloor  \frac{n}{2} \right\rfloor +1$, where
$d= \max \{m_1,\ldots,m_n\}$ is the degree of the curve.
On the dual side, we show that if a univariate polynomial $p(t)$ of
degree $d$ with at most $2k+1$ monomials is non-negative on  $\R_+$,
then $p$ admits a representation
$p = t^0 \sigma_0 + \cdots + t^{d-k} \sigma_{d-k}$, where the
polynomials $\sigma_0,\ldots,\sigma_{d-k}$ are sums of squares and $\deg (\sigma_i) \le 2k$. 
The latter is a univariate positivstellensatz for sparse polynomials,
with non-negativity of $p$ being certified by sos polynomials whose
degree only depends on the sparsity of~$p$.
Our results fit into the general attempt of
formulating polynomial
optimization problems as semidefinite problems with LMIs of small
size. Such small-size descriptions are much more tractable from a
computational viewpoint.
\end{abstract}


\section{Introduction}

Polynomial optimization studies the task
of minimizing an $n$-variate polynomial $p \in \R[x_1,\ldots,x_n]$
over a semialgebraic set $S \subseteq \R^n$. The convexification
approach to this problem consists of picking finite sets
$A\supset\supp(p)$ of $\Z_\plus^n$, considering the monomials
$x^\alpha=x_1^{\alpha_1}\cdots x_n^{\alpha_n}$ with $\alpha\in A$
and finding exact or approximate descriptions of the closed
convex hulls
\[
K=\overline{ \conv \setcond{ (x^\alpha)_{\alpha \in A}}{x \in S}}
\]
within some optimization paradigm. The basic idea is that optimizing
a polynomial
\(
p = c_0 + \sum_{\alpha \in A} c_\alpha x^\alpha \in \R[x_1,\ldots,x_n]
\)
means to optimize the affine-linear function
$\lambda : (y_\alpha)_{\alpha \in A}\mapsto c_0 + \sum_{\alpha \in A}
c_\alpha y_\alpha$
over $K$. Usually, descriptions of the sets $K$ arise from
\emph{positivstellens\"atze} from real algebra, since
non-negativity of $\lambda$ on $K$ corresponds to non-negativity of $p$
on $S$. Most positivstellensätze in real algebra employ
sum-of-squares (sos) certificates. This fact establishes  a connection
from real algebra  to semidefinite optimization, since the cone
$\Sigma_{n,2d}$ of $n$-variate sos polynomials of  degree $\le 2d$ is a
linear image of the cone of symmetric positive semidefinite matrices
of size $\binom{n+d}{n}$. As suggested by Lasserre in his seminal
work \cite{las2001} and following this line of ideas, Putinar's
positivstellensatz \cite{put} leads to a hierarchy of
outer approximations of $K$ in terms of semidefinite
constraints. Although Lasserre's approach provides a universal template
for converting polynomial problems to semidefinite ones, further
adjustment is usually needed
to make it computationally  tractable. Despite the fact that
semidefinite programs are solvable in polynomial time under mild
assumptions within a given error tolerance \cite{ali,nest:nem}, solving
a semidefinite program can quickly become extremely challenging in
practice since the size of the
semidefinite constraints is a critical parameter. We view a semidefinite
problem as the problem of optimization of a linear function subject
to finitely many linear matrix inequalities (LMIs). An LMI of size $d$
is the condition $M(y) \succeq 0$ that imposes semidefiniteness of
a symmetric $d \times d$ matrix $M(y)$ whose entries are affine-linear
functions in the variables $y = (y_1,\ldots,y_N)$.  To
get a first impression  of how the workload is increased by large
size of the LMIs, consider the approximation of a semidefinite problem
$\inf \{ L(y) \,:\, M_1(y) \succeq 0,\ldots, M_\ell(y) \succeq 0 \}$ by
the convex problem $L(y) -  \epsilon \sum_{i=1}^\ell \log \det M_i(y)$
that uses the logarithmic barrier and a small parameter $\epsilon>0$.
Solving the convex problem with the gradient descent method would involve
computating the gradients of the barriers, which requires to invert
the matrices $M_i(y)$ \cite[4.3.1]{bt:nem}.  But inverting
$M_i(y)$ of a large size is known to be expensive. So, while in
theoretical considerations it is customary to model $\ell$ LMIs as a single
LMI of size $\ell d$, using a block-diagonal matrix, this reduction
conceals the aspect of efficiency, since a general LMI of size $\ell d$
has a much higher computational cost than $\ell$ LMIs of size~$d$.

To address this issue, the following terminology was introduced in
\cite{av}:

\begin{dfn}
Let $C\subset\R^n$ be a set. A description of $C$ in the form
\begin{equation}
C = \setcond{ x \in \R^n}{ M_1(x,y) \succeq 0,\ldots, M_\ell(x,y)
\succeq 0 \ \text{for some} \ y \in \R^m} \label{sdp:descr}
\end{equation}
where the $M_i$ are LMIs in $(x,y)$, is called a \emph{lifted} (or
\emph{extended}) \emph{semidefinite representation}
of $C$. If every $M_i$ has size $\le d$, we say that \eqref{sdp:descr}
is a lifted representation
of size~$\le d$. The minimal $d$ for which $C$ admits a lifted
representation
of size $\le d$ is the \emph{semidefinite extension degree} of $C$,
denoted $\sxdeg(C)$. If no such $d$ exists one puts $\sxdeg(C)=\infty$.
\end{dfn}

By studying $\sxdeg(C)$ one keeps track of the size of the LMIs
needed to optimize linear functions over $C$, disregarding their
number. For computations, both size and number of the LMIs
play a role, but
size is more critical as a parameter. Clearly $\sxdeg(C) =1$
if and only if $C$ is a polyhedron, and it is easy to see that
$\sxdeg(C) \le 2$ if and only if $C$ is second-order cone
representable \cite{faw}. In \cite{av} it was shown that the
cone $\Sigma_{n,2d}$ of $n$-variate sos polynomials of degree $\le2d$
has $\sxdeg=\binom{n+d}{n}$, matching the size of the LMIs in
Lasserre's approach.
The rapid growth of $\binom{n+d}{n}$ in $n$ and $d$
explains why the computational time needed to solve Lasserre's
relaxations is extremely sensitive to the choice of $n$ and $d$. This issue
with the size of the LMIs in was detected by a number of researchers. Some
suggestions for how to cope with it in practice were made
in \cite{ahm,av2,wang1,wang}.  The common general idea in
\cite{ahm,av2,wang1,wang}  is to deliberately choose a cone $C$,
which has small $\sxdeg(C)$ by its construction, to serve as a
tractable outer approximation  of $K$. Or, from the dual viewpoint,
to choose a cone $P$ of non-negative polynomials of more specific
structure than in Putinar's positivstellensatz, such that $P$ has
small $\sxdeg(P)$. While these approaches seem to help in practice,
they are purely heuristic, since positivstellens\"atze are still
missing that would guarantee that positivitiy can indeed always be
certified in the intended, computationally less expensive, manner.

In this paper, we make a first step in filling this gap by studying
the size of semidefinite representations for sparse polynomial
optimization problems in one variable.
Consider a set $A = \{m_1,\ldots,m_n\}$ of positive integers
$m_1 > \cdots > m_n > 0$ and a non-degenerate interval
$I \subseteq \R$. For optimizing the univariate polynomial
$p=\sum_{a \in A} c_a t^a \in \R[t]$ on the interval $I$, we are
interested in finding a lifted semidefinite representation of
$K  = \overline{ \conv \{(t^{m_1},\dots,t^{m_n})\colon t\in I\}}
\subseteq \R^n$.
Since the curve $(t^{m_1},\ldots,t^{m_n})$ is a projection of  the
rational normal curve $(t^i)_{1 \le i \le m_1}$ of degree $m_1$,
Lasserre's approach gives a description of size $\left\lfloor \frac{m_1}{2} \right\rfloor+1$. When the
number $n=|A|$ of monomials is small compared to the degree $m_1$,
it is desirable to find an alternative
description of smaller size. Our main result (Theorem~\ref{thm:mon:curve})
shows that the semidefinite extension degree of such $K$ is at most
$\left\lfloor \frac{n}{2} \right\rfloor  + 1$, which is the best
possible bound. Consequently, the size of the descripttion depends
only on the number of monomials, and not on the degree. The description
is completely explicit and follows from a sparse positivstellensatz
(Theorem~\ref{moncasexplicit}). The latter characterizes non-negativity
of a degree~$d$ polynomial $p$ with at most $2k+1$ terms
on an interval $I \subset \R_\plus$, by using the cone
$t^0 \Sigma_{1,2k} + \cdots + t^{d-2 k} \Sigma_{1,2k}$, which is of
semidefinite extension degree $k+1$. The key technical ingredient
in our proof is Jacobi's bialternant formula for Schur polynomials from the theory of symmetric polynomials. 

Our result about non-negativity of univariate sparse polynomials
should help to understand the impact of sparsity-based
approaches to optimization problems of arbitrary dimension. Univariate
problems may, on the one hand, demonstrate phenomena that
occur in every dimension. On the other,
whenever some limitation can be identified in the one-dimensional
case, it will also be present in a certain form in an $n$-dimensional
setting as well. Furthermore, our main result can serve as a starting
point for further studies in
several directions. It would obviously be interesting to obtain
similar results in the multivariate case. While our $K$ is
given as the closed convex hull of a monomial curve, one may ask for bounds
on the semidefinite extension degree of closed convex hulls $K$ of
arbitrary semialgebraic curves $S \subset \R^n$. For $n=2$, it was
proved in \cite{sch:sxdeg} that every closed convex \sa\ set $K$ in
the plane is second-order cone representable, i.e.\ has
$\sxdeg(K)\le2$.
It can be shown that the bound $\sxdeg(K) \le
\left\lfloor \frac{n}{2} \right\rfloor  + 1$ holds
for the closed convex hull of an arbitrary semialgebraic curve
$S \subseteq\R^n$, but the proof is much more involved
\cite{sch:gencase}.

The paper is organized as follows. After a few preliminaries in
Section~\ref{prelim}, the main result is obtained in
Section~\ref{conv:hull:curves}. In Section~\ref{sect:sparse} we
introduce the cones of sums of copositive fewnomials, which form a
sparse counterpart of the cones of non-negative polynomials. We
show that in the one-dimensional case these cones admit semidefinite
descriptions with LMIs of small size, and explain that such
descriptions lead to variations of Lasserre's relaxations that are
based on LMIs of small size.


\section{Preliminaries} \label{prelim}

\begin{lab}
Let $\Z_+ = \{0,1,2,\ldots\}$ and $\N = \{1,2,3,\ldots\}$ and
$\R_+ = [0,\infty[$. The
cardinality of a set $A$ is denoted by $|A|$. For a tuple
$\alpha  = (\alpha_1,\dots,\alpha_n)\in \Z^n$ we use the notation
$|\alpha| := \alpha_1 + \cdots + \alpha_n$. For univariate polynomials
we mostly use $t$ to denote the variable, and we write
$\R[t]_d = \setcond{f \in \R[t] }{ \deg(f) \le d }$.
The \emph{support} of a polynomial $p = \sum_{\alpha} c_\alpha x^\alpha
\in \R[x_1,\ldots,x_n]$ is the set
$\supp (p) := \setcond{ \alpha \in \Z_+^n}{c_\alpha \ne 0}$.
By $\sfS^m$ we denote the space of real symmetric $m\times m$ matrices,
and $\sfS_\plus^m$ is the cone of positive semidefinite matrices in
$\sfS^m$. The linear span of a subset $M$ of a vector space is denoted
$\lin(M)$.
\end{lab}

\begin{lab}\label{faces}%
Let $C\subset\R^n$ be a convex cone. The \emph{dual cone} of $C$ is
$C^*=\{y\in\R^n\colon\all x\in C$ $\bil xy\ge0\}$, where $\bil xy$
denotes the standard inner product on $\R^n$. The bi-dual
$C^{**}:=(C^*)^*$ of $C$ equals the closure of $C$, i.e.\
$C^{**}=\ol C$. The cone $C$ is \emph{pointed} if $C\cap(-C)=\{0\}$.
When $C\subset V$ is a convex cone in an arbitrary real vector space
$V$, the dual cone is $C^\ast := \setcond{ y \in V^\vee}
{\forall x \in C \ y(x) \ge 0} $, where  $V^\vee$ is the dual vector
space of~$V$.

A \emph{face} of $C$ is a convex cone $F$ with $F\subset C$ such that
$x,\,y\in C$ and $x+y\in F$ imply $x,\,y\in F$.
For any $x\in C$ there is a unique inclusion-minimal face $F$ of $C$
with $x\in F$, called the \emph{supporting face} of~$x$ in~$C$.
One-dimensional faces of convex cones are called \emph{extreme rays}.
It is well-known that a finite-dimensional closed and pointed convex
cone is the Minkowski sum of its extreme rays.
\end{lab}

\begin{lab}
We briefly explain the conic duality behind the approaches in polynomial optimization by providing a generic version of the discussion in \cite{las} Ch.~10. We are given a subset $S$ of $\R^n$ and polynomials $p,q$ in a finite-dimensional vector subspace $V$ of $\R[x_1,\ldots,x_n]$ such that $q>0$ on $S$.
A common choice for $q$ is the constant  $q=1$. The problem of minimizing the quotient $p/q$ over $S$ can be relaxed (i.e.,  lower bounded) by making use of a closed convex cone $C \subseteq V$ that satisfies $g \ge 0$ on $S$ for every $g \in C$. One has
\[
\inf_S \frac{p}{q} \ge \sup \setcond{ \lambda \in \R}{ p - \lambda q \in C },
\]
where the supremum is a conic optimization problem dual to the problem
\[
\inf \setcond{ v(p) }{  v \in C^\ast, \ v(q) =1}.
\]
In the following proposition, the mentioned duality is phrased without any reference to polynomial by identifying $V$ and its dual space $V^\vee$ with $\R^N$, where $N := \dim(V)$.
\end{lab}

\begin{prop}\label{conic:opt:duality}
	Let $N$ be a positive integer and $C \subseteq \R^N$ be a closed and pointed convex cone. Then for every $p \in C$ and $q \in C \setminus \{ 0\}$ one has
	\begin{equation} \label{conicduality}%
		\sup \setcond{\lambda \in \R}{p - \lambda q \in C} = \inf \setcond{ \sprod{p}{v} }{v \in C^\ast, \ \sprod{v}{q} = 1}.
	\end{equation}
\end{prop}

This is a special case of the duality of conic optimization problems,
cf.\ the discussion in \cite{ble} 2.1.4. We omit the proof which is
quite straightforward.



\section{Convex hulls of monomial curves} \label{conv:hull:curves}

Let $I\subset\R$ be a non-degenerate closed interval, let
$V\subset\R[t]$ be a linear subspace of finite dimension,
and let $P=\{f\in V\colon f\ge0$ on $I\}$, a closed and pointed
convex cone in $V$.
We start by showing that every face of the cone $P$ is described by
suitable vanishing conditions at points of $I$, or at infinity when
$I$ is unbounded. Let $\ord_s(f)$ denote the order of vanishing of
$f\in\R[t]$ at $s\in\R$.

\begin{prop}\label{extstrahl}%
For $0\ne f\in P$ put $W_f=\{g\in V\colon\all s\in I$
$\ord_s(g)\ge\ord_s(f)\}$. Let $U_f=\{g\in W_f\colon\deg(g)\le
\deg(f)\}$ if $I$ is unbounded, and put $U_f=W_f$ if $I$ is
compact. Then $U_f$ is the linear span of the supporting face
$F_f$ of $f$. In particular, $\dim(F_f)=\dim(U_f)$.
\end{prop}

For extreme rays of $P$, this implies:

\begin{cor}\label{xrayrootscpt}%
Assume that $I$ is compact, let $f\in V$ span an extreme ray of $P$.
Then $f$ has at least $\dim(V)-1$ roots in $I$, counting with
multiplicity.
\end{cor}

\begin{proof}
By assumption we have $F_f=\R_\plus f$, so $W_f=\R f$ by
Proposition \ref{extstrahl}. Since $W_f$ consists of the elements in
$V$ with at least the same roots in $I$ as $f$, there have to be at
least $\codim(W_f)=\dim(V)-1$ many roots of $f$ in $I$, counting with
multiplicities.
\end{proof}

In the non-compact case, the result reads as follows:

\begin{cor}\label{xrayrootsnoncpt}%
Let $I$ be unbounded, let $f\in P$ span an extreme ray of $P$.
Write $d=\deg(f)$ and $V_d:=\{g\in V\colon\deg(g)\le d\}$. Then $f$
has at least $\dim(V_d)-1$ roots in $I$, counting with multiplicity.
\end{cor}

\begin{proof}
Same argument as for Corollary \ref{xrayrootscpt}, since $\R f$
consists of the elements in $V_d$ with at least the same roots in $I$
as~$f$.
\end{proof}

\begin{proof}[Proof of Proposition \ref{extstrahl}]
The supporting face of $f$ is $F_f=\{p\in P\colon\ex\gamma>0$ with
$f-\gamma p\in P\}$.
Clearly, if $p,\,q\in P$ with $f=p+q$ then $p,\,q\in W_f$, and also
$\deg(p)$, $\deg(q)\le\deg(f)$ when $I$ is unbounded. Therefore
$F_f\subset U_f$ holds. To prove $U_f\subset\lin(F_f)$, note that for
every $g\in U_f$ there exists a constant $c>0$ such that $c|g|\le f$
on $I$. Indeed, for every $s\in I$ there is a constant $c_s>0$ with
$c_s|g(t)|\le f(t)$ on some neighborhood $J_s$ of $s$. If $I$
contains a right half-line, there are $a,\,c>0$ with
$c|g(t)|\le f(t)$ for all $t\in J_\infty=\left[a,\infty\right[$,
and similarly in the case of a left half-line. By passing to a
finite subcovering of $I$ we find a constant $c$ as required.
	
By the preceding remark we find a linear basis $g_1,\dots,g_r$ of
$U_f$ with the property that $|g_i|\le f$ on $I$ for $i=1,\dots,r$.
It follows that $f-g_i\in F_f$ ($i=1,\dots,r$),
and so $U_f=\lin(f-g_1,\dots,f-g_r,\,f)$ is contained in
$\lin(F_f)$.
\end{proof}

We continue to assume that $V\subset\R[t]$ is a linear subspace of
finite dimension $\dim(V)=n+1$. Assume we are given $n$ vanishing
conditions for elements of $V$ at points on the real line, where we
allow conditions of higher order vanishing. If these conditions are
independent on $V$, the following lemma gives an explicit formula for
the essentially unique solution of these conditions in $V$:

\begin{lem}\label{simplema}%
Let $p_0,\dots,p_n$ be a linear basis of $V$. Let
$\xi=(\xi_1,\dots,\xi_r)$ be a tuple of pairwise different real
numbers, and let $b_1,\dots,b_r\ge1$ be integers with
$\sum_{i=1}^rb_i=n$. The following conditions are equivalent:
\begin{itemize}
\item[(i)]
The matrix
\begin{equation}\label{bigmatbieq2}%
A(t;\xi)\ =\ \begin{pmatrix}p_0(t)&\cdots&p_n(t)\\
p_0(\xi_1)&\cdots&p_n(\xi_1)\\
\vdots&&\vdots\\
p_0^{(b_1-1)}(\xi_1)&\cdots&p_n^{(b_1-1)}(\xi_1)\\
\vdots&&\vdots\\
p_0(\xi_r)&\cdots&p_n(\xi_r)\\
\vdots&&\vdots\\
p_0^{(b_r-1)}(\xi_r)&\cdots&p_n^{(b_r-1)}(\xi_r)
\end{pmatrix}
\end{equation}
of size $(n+1)\times(n+1)$ and with entries in $\R[t]$ has
non-zero determinant;
\item[(ii)]
the subspace $\{f\in V\colon\ord_{\xi_i}(f)\ge b_i$ for
$i=1,\dots,r\}$ of $V$ has dimension one.
\end{itemize}
When (i) and (ii) hold, the polynomial $\det A(t;\xi)$ is the unique
(up to scaling) element of $V$ that vanishes in $\xi_i$ of order at
least $b_i$, for $i=1,\dots,r$.
\end{lem}

\begin{proof}
Note that the dimension in (ii) is always $\ge1$.
The matrix $B$ formed by the lower $n$ rows of \eqref{bigmatbieq2} is
the matrix of the linear map $\phi\colon V\to\R^n$,
$$p\ \mapsto\ \Bigl(p(\xi_1),\dots,p^{(b_1-1)}(\xi_1),\dots,
p(\xi_r),\dots,p^{(b_r-1)}(\xi_r)\Bigr)$$
with respect to the basis $p_0,\dots,p_n$ of $V$. Subspace (ii) is
just the kernel of $\phi$. The determinant of \eqref{bigmatbieq2} is
an element of $V$, and is non-zero if and only if $B$ has a
non-vanishing $n\times n$-minor.
This is equivalent to $\phi$ being surjective, and hence also to
(ii). The last assertion is clear since $\det A(t;\xi)$ is an
element of $V$ that has a zero of multiplicity $\ge b_i$ at $\xi_i$,
for $i=1,\dots,r$.
\end{proof}

\begin{lab}\label{schurpoly}%
Next we recall some background on Schur polynomials.
Let $n\in\N$ be a fixed integer and let $x=(x_0,\dots,x_n)$ be a
tuple of indeterminates. We consider partitions
$\mu=(m_0,\dots,m_n)$ into $n+1$ pieces, with
$m_0\ge\cdots\ge m_n\ge0$.
A particular partition is $\delta=(n,n-1,\dots,1,0)$.
Given a partition $\mu$ as above, the determinant
$$F_\mu(x)\>:=\>\det\begin{pmatrix}x_0^{m_0}&\cdots&x_0^{m_n}\\
\vdots&&\vdots\\x_n^{m_0}&\cdots&x_n^{m_n}\end{pmatrix}$$
is identically zero unless the $m_i$ are pairwise distinct. In this
latter case,
$\lambda=\mu-\delta=\bigl(m_0-n,\dots,m_{n-1}-1,m_n\bigr)$ is
another partition.
Clearly, the Vandermonde product
$$v(x)\>:=\>\prod_{0\le i<j\le n}(x_i-x_j)\>=\>F_\delta(x)$$
divides $F_\mu(x)$.
The co-factor is the \emph{Schur polynomial}
$s_\lambda(x)$ of the partition $\lambda$. In other words, the
\emph{bialternant formula}
\begin{equation}\label{bialtfml}%
F_\mu(x)\>=\>v(x)\cdot s_\lambda(x)
\end{equation}
holds. Depending on how Schur polynomials are introduced, identity
\eqref{bialtfml} is either the definition of $s_\lambda(x)$ or a
theorem, see \cite{st} Theorem 7.15.1.

The Schur polynomial $s_\lambda(x)$ is symmetric as a polynomial in
the variables $x_i$, homogeneous of degree $|\lambda|$, and of degree
$\lambda_0$ with respect to each variable $x_i$.
Schur polynomials have the remarkable property that all their
coefficients are non-negative integers. In fact there exists a
combinatorial description of the coefficients, see Section 7.10 in
\cite{st}. In our context, the integrality of the coefficients plays no role, but the non-negativity is a crucial property. 
\end{lab}

\begin{lab}\label{bigmatbi}%
We use Schur polynomials to deduce a product formula for the
determinant in Lemma \ref{simplema}, in the case where the $p_i$
are monomials. Let $\mu=(m_0,\dots,m_n)$ be a partition into
different parts, i.e.\ with $m_0>\cdots>m_n\ge0$. Write
$p_i(t)=t^{m_i}$ for $i=0,\dots,n$, and denote the $j$-th derivative
of $p_i(t)$ by $p_i^{(j)}(t):=\frac{d^j}{dt^j}p(t)$ ($j\ge0$).
Moreover we introduce the tuples
$p(t)=\bigl(p_0(t),\dots,p_n(t)\bigr)$ and
$p^{(j)}(t)=\bigl(p_0^{(j)}(t),\dots,p_n^{(j)}(t)\bigr)$ for $j\ge0$.

Let $b=(b_0,\dots,b_r)$ be a tuple of integers $b_i\ge1$ such that
$\sum_ib_i=n+1$, let $y=(y_0,\dots,y_r)$ be a tuple of $r+1$
variables. We consider the determinant $F_{\mu,b}(y)$ of size $n+1$
that contains, for each $i=0,\dots,r$, the rows
$$p^{(j)}(y_i)\>=\>\bigl(p_0^{(j)}(y_i),\dots,p_0^{(j)}(y_i)\bigr)$$
for $j=0,1,\dots,b_i-1$. In other words, let
\begin{equation}\label{bigmatbieq}%
F_{\mu,b}(y)\>:=\>\det\begin{pmatrix}p_0(y_0)&\cdots&p_n(y_0)\\
p_0'(y_0)&\cdots&p_n'(y_0)\\
\vdots&&\vdots\\
p_0^{(b_0-1)}(y_0)&\cdots&p_n^{(b_0-1)}(y_0)\\
\vdots&&\vdots\\
p_0(y_r)&\cdots&p_n(y_r)\\
\vdots&&\vdots\\
p_0^{(b_r-1)}(y_r)&\cdots&p_n^{(b_r-1)}(y_r)\\
\end{pmatrix}
\end{equation}
and let again $\lambda=\mu-\delta$.
\end{lab}

\begin{prop}\label{bialtmultroots}%
The determinant $F_{\mu,b}(y)$ has the product decomposition
$$F_{\mu,b}(y)\>=\>c\cdot v_b(y)\cdot s_{\lambda,b}(y)$$
where
$$v_b(y)\>:=\>\prod_{0\le i<j\le r}(y_i-y_j)^{b_ib_j}$$
and
$$s_{\lambda,b}(y)\>:=\>
s_\lambda\bigl(\underbrace{y_0,\dots,y_0}_{b_0},\>\dots,\>
\underbrace{y_r,\dots,y_r}_{b_r}\bigr)$$
and $c$ is a constant, equal to
$$c\>=\>\prod_{i=0}^r(-1)^{b_i(b_i-1)/2}\cdot(b_i-1)!$$
\end{prop}

\begin{proof}
We inductively derive the assertion from the bialternant formula.
Let $A_0$ be the matrix with rows $p(x_0),\dots,p(x_n)$, so
$$\det(A_0)\>=\>s_\lambda(x)\cdot\prod_{0\le i<j\le n}(x_i-x_j)$$
by \eqref{bialtfml}. Replacing the second row $p(x_1)$ by
$(p(x_1)-p(x_0))/(x_1-x_0)$, the determinant gets divided by
$x_1-x_0$. If we now specialize $x_1:=x_0$, the resulting matrix
$A_1$ has second row $p'(x_0)$ and has determinant
$$\det(A_1)\>=\>-s_\lambda(x_0,x_0,x_2,\dots,x_n)\cdot
\prod_{j=2}^n(x_0-x_j)^2\cdot\prod_{2\le i<j\le n}(x_i-x_j)$$
By pulling out $x_1-x_0$, we have therefore used up the factor
$x_0-x_1$ from $v(x)$ and got a minus sign. Now iterate this step.
Next we replace the third row of $A_1$, which is $p(x_2)$, by
$$\frac{p(x_2)-p(x_0)-(x_2-x_0)p'(x_0)}{(x_2-x_0)^2}$$
Specializing $x_2:=x_0$, the resulting matrix $A_2$ has rows
$$p(x_0),\ p'(x_0),\ \frac12p''(x_0),\ p(x_3),\dots$$
and has determinant
$$\det(A_2)\>=\>-s_\lambda(x_0,x_0,x_0,x_3,\dots,x_n)\cdot
\prod_{j=3}^n(x_0-x_j)^3\cdot\prod_{3\le i<j\le n}(x_i-x_j)$$
And so on. After $b_0$ many steps, the rows have become
$$p(x_0),\ p'(x_0),\ \dots,\ \frac1{(b_0-1)!}p^{(b_0-1)}(x_0),\
p(x_{b_0}),\dots,p(x_n),$$
and at that point we have thrown in $0+1+\cdots+(b_0-1)=\binom{b_0}2$
many minus signs.

We can now repeat this procedure. The next step consists in doing
$b_1$ many steps on the variables $x_{b_0},\dots,x_{b_0+b_1-1}$.
And so on. Finally, relabel the variables that have survived as
$y_0,\dots,y_r$.
\end{proof}

\begin{lab}\label{setupmoncase}%
After these preparations we come to the main result of our paper.
Let $\mu=(m_0,\dots,m_n)$ with $m_0>\cdots>m_n\ge0$, write
$k=1+\lfloor\frac n2\rfloor$. Let $\R[t]$ be the ring of polynomials
in the variable~$t$. We consider the subspace $V=V_\mu$ of $\R[t]$
that is spanned by the monomials $t^{m_0},\dots,t^{m_n}$. Let
$I\subset\R$ be an interval, and let
$$P\>:=\>\{f\in V\colon f\ge0\text{ on }I\}.$$
Note that $P$ is a closed convex cone in $V$. The key result is the
following nichtnegativstellensatz:
\end{lab}

\begin{thm}\label{moncasexplicit}%
Let (a) $I=[0,1]$ or (b) $I=\R_\plus=\left[0,\infty\right[$. Then
every $f\in P$ can be written as a finite sum
\begin{align}\label{moncasefml}%
\begin{split}
{\rm(a)}\quad f & =\ \sum_ig_i(t)^2p_i(t)+(1-t)\sum_jh_j(t)^2q_j(t), \\
{\rm(b)}\quad f & =\ \sum_ig_i(t)^2p_i(t),
\end{split}
\end{align}
where $g_i,p_i,h_j,q_j$ are polynomials with
$\deg(g_i),\,\deg(h_j)\le\lfloor\frac n2\rfloor=k-1$, such that the
degree of every summand is $\le m_0$, and such that the coefficients
of $p_i,\,q_j$ are all non-negative.
\end{thm}

Conversely, when $f$ has a representation as in the theorem, it is
obvious that $f\ge0$ on $I$.

\begin{lab}
To start the proof of Theorem \ref{moncasexplicit} (in either case
(a) or~(b)), let $Q$ be the set of all $f\in V$ that have a
representation \eqref{moncasefml}. Then $Q$ is a convex cone, and
$Q\subset P$.
Let $E\subset P$ be the set of all polynomials $f$ that generate an
extreme ray of $P$.
Every element of $P$ is a sum of finitely many elements of $E$,
since the cone $P$ is closed and pointed. To prove $Q=P$, it
therefore suffices to show $f\in Q$ for any given $f\in E$.
\end{lab}

\begin{lab}\label{reductofproof}%
Let us make two more reduction steps.
Assume that $f\in Q$ has been shown (for all monomial subspaces)
whenever $f\in E$ satisfies $f(0)>0$. Let $f\in E$ with $f(0)=0$.
Writing $f=t^w\tilde f$ where $w=\ord_0(f)$, there is an index
$0\le s\le n$ with $m_s=w$.
The subspace $\tilde V=\lin(t^{m_0-w},\dots,t^{m_{s-1}-w},1)$ of
$\R[t]$ has dimension $s+1\le n+1$ and contains $\tilde f$. Moreover
$\tilde f\ge0$ on $I$ and $\tilde f(0)>0$. To fix ideas, assume we
are in case~(a), so $I=[0,1]$. By assumption, the theorem holds for
$\tilde f$ and the subspace $\tilde V$. This means that there is an
identity
\begin{equation}\label{inductfml}%
	\tilde f\>=\>\sum_ig_i(t)^2p_i(t)+(1-t)\sum_jh_j(t)^2q_j(t)
\end{equation}
where $\deg(g_i),\,\deg(h_j)\le\lfloor\frac s2\rfloor$ and the
$p_i,\,q_j$ have non-negative coefficients, such that every summand
in \eqref{inductfml} has degree $\le m_0-w$. Multiplying the identity
with $t^w$, we get an identity for $f$ as claimed in Theorem
\ref{moncasexplicit}(b).
In case~(b), the argument is exactly the same.
\end{lab}

\begin{lab}\label{degreduct}%
So we need to establish an identity \eqref{moncasefml} for every
$f\in E$ with $f(0)>0$. For this, we clearly may discard all
monomials of degree greater than $d=\deg(f)$. In other words, we may
assume that $\deg(f)=m_0$.
This reduction step plays a role only in case~(b). (In fact,
$\deg(f)=m_0$ is automatic if $I=[0,1]$, as will be seen from the
proof below.)
\end{lab}

\begin{lab}\label{genaunnst}%
From Descartes' rule of signs (e.g.\ \cite{ks} Cor.~1.10.3)
it follows that every non-zero $f\in V$ has at most $n$ strictly
positive roots, counting with multiplicity.
On the other hand, Corollaries \ref{xrayrootscpt} and
\ref{xrayrootsnoncpt} show that every $f\in E$ with $\deg(f)=m_0$
has at least $n$ roots in $I$. If in addition $f(0)>0$, it follows
that $f$ has precisely $n$ strictly positive roots, and they all lie
in~$I$.
\end{lab}

\begin{lab}
Let $f\in E$ with $f(0)>0$ and $\deg(f)=m_0$. Note that $m_n=0$ now
since $f(0)\ne0$. By \ref{genaunnst}, $f$ has precisely $n$ positive
roots, and they all lie in~$I$. It suffices to show that every such
$f$ satisfies an identity \eqref{moncasefml}.

Let $\xi_1,\dots,\xi_r$ be the different positive roots of $f$, and
let $b_i=\ord_{\xi_i}(f)$ ($i=1,\dots,r$). Each $b_i$ is an even
integer, except possibly in case~(a) when $\xi_i=1$. We have
$\sum_{i=1}^rb_i=n$ since $f$ has $n$ roots in~$I$.

Consider the determinant \eqref{bigmatbieq} in \ref{bigmatbi}, with
$b_0:=1$ and $p_i=t^{m_i}$ ($i=0,\dots,n$). After substituting
$y_0=t$ and $y_i=\xi_i$ ($i=1,\dots,r$), Proposition
\ref{bialtmultroots} shows that the determinant has a factorization
$$F_{\mu,b}(t,\xi_1,\dots,\xi_r)\>=\>
\gamma\prod_{i=1}^r(t-\xi_i)^{b_i}\cdot
s_\lambda\bigl(t,\>\underbrace{\xi_1,\dots,\xi_1}_{b_1},\>\dots,\>
\underbrace{\xi_r,\dots,\xi_r}_{b_r}\bigr)$$
with $\gamma\ne0$ a constant. The last factor
$s_\lambda(t,\xi_1,\dots,\xi_r)$ is a polynomial in~$t$. Since all
coefficients of the Schur polynomial $s_\lambda$ are $\ge0$, and
since $\xi_i>0$ for all $i$, this polynomial is not identically zero
(and has degree $\lambda_0=m_0-n$). Hence the determinant is
non-zero, and Lemma \ref{simplema} implies that it agrees with $f$ up
to scaling.
This means that $f$ has a factorization
\begin{equation}\label{fidentitat}%
f\>=\>\gamma'\prod_{i=1}^r(t-\xi_i)^{b_i}\cdot
s_\lambda(t,\underbrace{\xi_1,\dots,\xi_1}_{b_1},\dots,
\underbrace{\xi_r,\dots,\xi_r}_{b_r})
\end{equation}
with a constant $\gamma'\ne0$. The first factor has the form
$$\prod_{i=1}^r(t-\xi_i)^{b_i}\>=\>g(t)^2u(t)$$
where $\deg(g)=k-1=\lfloor\frac n2\rfloor$ and $u(t)=1$ or $t-1$.
The case $u(t)=t-1$ occurs only in case~(a) ($I=[0,1]$) when $n$ is
odd.
The last factor $s_\lambda(t,\xi_1,\dots)$ in \eqref{fidentitat} is a
polynomial in $t$ with non-negative coefficients.
So, up to a non-zero constant factor, the right hand side has the
form claimed in Theorem \ref{moncasexplicit}.
Since $f|_I\ge0$ we see that $\gamma'>0$ in case~(a) with $n$ even, or
in case~(b). When $n$ is odd in case~(a), we have $\xi_i=1$ for
some~$i$, and $b_i$ is odd, so $\gamma'<0$ in this case.
The theorem is proved.
\qed
\end{lab}

Let $\Sigma=\Sigma_{2k}:=\{g\in\R[t]\colon\deg(g)\le2k$, $g\ge0$ on
$\R\}$. Each $g\in\Sigma_{2k}$ can we written $g=g_1^2+g_2^2$ with
polynomials $g_1,\,g_2$ of degree $\le k$. Theorem
\ref{moncasexplicit} can be stated in the following alternative form:

\begin{cor}\label{cor2moncasthm}%
Again let (a) $I=[0,1]$ or (b) $I=\R_\plus$. With assumptions as in
Theorem \ref{setupmoncase}, the inclusion
\begin{itemize}
\item[(a)]
$P\>\subset\>\Bigl(\Sigma+t\Sigma+\cdots+t^{m_0-2k}\Sigma\Bigr)+
(1-t)\Bigl(\Sigma+t\Sigma+\cdots+t^{m_0-2k-1}\Sigma\Bigr)$
\end{itemize}
or
\begin{itemize}
\item[(b)]
$P\>\subset\>\Bigl(\Sigma+t\Sigma+\cdots+t^{m_0-2k}\Sigma\Bigr)$
\end{itemize}
holds, respectively.
\end{cor}

\begin{rem}\label{sparsepss}%
Part (b) of Corollary \ref{cor2moncasthm} can be considered to be a
\emph{sparse} (or \emph{fewnomial}) \emph{positivstellensatz} (or
rather, \emph{nichtnegativstellensatz}) for univariate polynomials:
Every
polynomial $f\in\R[t]$ with $m$ monomials that is non-negative on
$\R_\plus$ can be written as a (finite) sum
$$f\>=\>\sum_{i\ge0}t^i\bigl(p_i(t)^2+q_i(t)^2\bigr)$$
where $p_i,\,q_i$ are polynomials of
degree~$\le\lfloor\frac{m-1}2\rfloor$. This point of view will be
expanded in more detail in the next section.
\end{rem}

\begin{rem}\label{semidefblockrep}%
As before, let $V\subset\R[t]$ be a linear subspace generated by
$n+1$ monomials, and let $P=\{f\in V\colon f\ge0$ on~$I\}$ where
(a)~$I=[0,1]$ or (b)~$I=\R_\plus$. In either case, Corollary
\ref{cor2moncasthm} allows us to read off a block semidefinite
representation of $P$ of block size at most $k+1$. Let
$\varphi\colon\sfS^{k+1}\to\R[t]_{2k}$ be the linear map that
sends a symmetric matrix $M=(a_{ij})_{0\le i,j\le k}$ to
$$\varphi(M)\>:=\>(1,t,\dots,t^k)\cdot M\cdot(1,t,\dots,t^k)^\top \>=\>
\sum_{i,j=0}^ka_{ij}t^{i+j}$$
Then $\varphi(\sfS^{k+1}_\plus)=\Sigma_{2k}$. In case~(a), consider
the linear map
$$\phi\colon\bigl(\sfS^{k+1}\bigr)^{2(m_0-2k)+1}\>\to\>\R[t]$$
given by
$$\bigl(M_0,\dots,M_{m_0-2k};\,N_0,\dots,N_{m_0-2k-1}\bigr)\>\mapsto\>
\sum_{i=0}^{m_0-2k}t^i\varphi(M_i)+(1-t)\sum_{j=0}^{m_0-2k-1}
t^jS\varphi(N_j)$$
For (b), consider the linear map
$\phi\colon\bigl(\sfS^{k+1}\bigr)^{m_0-2k+1}\>\to\>\R[t]$
given by
$$\bigl(M_0,\dots,M_{m_0-2k}\bigr)\>\mapsto\>
\sum_{i=0}^{m_0-2k}t^i\varphi(M_i)$$
In either case we have $P=V\cap\im(\phi)$, according to Corollary
\ref{cor2moncasthm}. This is an explicit block semidefinite
representation of $P$ of block size $k+1$.
\end{rem}

\begin{rem}\label{semidefblockrep1}%
For sparse polynomials $f$ in $\R[t]$ that are non-negative on the whole
real axis, there does not in general exist a sparse decomposition
\[
f = \sum_{i\in 2 \Z_+} t^i \left(p_i(t)^2 + q_i(t)^2 \right)
\]
similar to the one described in Remark~\ref{sparsepss}.
For example, $f=3t^4-4t^3+1=(t-1)^2(3t^2+2t+1)$ is non-negative on
$\R$ but cannot be written $f=g_1(t)+t^2 g_2(t)$
with $g_1,\,g_2$ sums of squares of linear polynomials.

Still, we may easily produce a block semidefinite representation of
block size $k+1$ for $P=\{f\in V\colon f\ge0$ on~$\R\}$, from the
case $I=\R_\plus$ in \ref{semidefblockrep}. It suffices to remark that
$f(t)\ge0$ on $\R$ is equivalent to $f(t)$ and $f(-t)$ both being
$\ge0$ on $\R_\plus$.
\end{rem}

By cone duality we get the desired theorem for convex hulls of
monomial space curves:

\begin{thm}\label{thm:mon:curve}
Let $m_1,\dots,m_n\ge1$ be integers, and let $I\subset\R$ be a \sa\
set. The closed convex hull $K$ of the set
$$S\>=\>\bigl\{(t^{m_1},\dots,t^{m_n})\colon t\in I\bigr\}$$
in $\R^n$ has semidefinite extension degree at most
$\lfloor\frac n2\rfloor+1$.
\end{thm}

\begin{proof}
The set $K$ is identified with an affine-linear slice of the dual
cone $P^*$ of $P$. Since $\sxdeg(P^*)=\sxdeg(P)$
(\cite{sch:sxdeg} Prop.~1.7), it suffices to prove
$\sxdeg(P)\le\lfloor\frac n2\rfloor+1$.
We may assume that the $m_i$ are pairwise distinct. When
$I=[0,1]$ or $I=\R_\plus$, the claim
$\sxdeg(P)\le\lfloor\frac n2\rfloor+1$ was shown in Remark
\ref{semidefblockrep}. We sketch how the case of other sets $I$ can
essentially be reduced to these two cases, without going into full
details. Clearly $I$ can be assumed to be closed. If $I=I_1\cup I_2$
with $I_1,\,I_2$ \sa, it is enough to prove the claim for both $I_1$
and $I_2$, in view of \cite{sch:sxdeg} Prop.~1.6.
In this way we reduce to considering $I=[a,b]$ or
$I=\left[a,\infty\right[$ where $0\le a<b<\infty$. The cases $a=0$
are already done (assuming $b=1$ in the compact case was nowhere
essential). Assume $I=[a,b]$ with $0<a<b<\infty$. Then every $f\in P$
that generates an extreme ray of $P$ has exactly $n$ roots in $I$,
counting with multiplicities. So the reduction step
\ref{reductofproof} in the proof of Theorem \ref{moncasexplicit} is
not needed. Otherwise we may just follow the proof of this theorem.
In this way we arrive at a representation of every element of $P$
in a form similar to \ref{moncasexplicit}(a), but with weights
$1$, $t-a$, $b-t$ and $(t-a)(b-t)$ instead of only $1$ and $1-t$.
When $I=\left[a,\infty\right[$ with $a>0$, we may proceed in a
similar way.
\end{proof}

\begin{rem}
For $I=[0,1]$ or $I=\R_\plus$, an explicit block semidefinite
representation of $K$ of block size $k=\lfloor\frac n2\rfloor+1$
can be obtained from Remarks \ref{semidefblockrep},
\ref{semidefblockrep1} by dualizing.
\end{rem}

\begin{rem}
It is natural to ask whether the bound
$\sxdeg(K)\le\lfloor\frac n2\rfloor+1$ extends to cases more general
than convex hulls of monomial curves. In fact, the same bound is true
in general whenever $K\subset\R^n$ is the closed convex hull of a
one-dimensional semialgebraic set in $\R^n$ \cite{sch:gencase}.
However the proof gets much more difficult than in the monomial case.
Even when the curve is parametrized by polynomials instead of
monomials, one has to argue locally on sufficiently small intervals
on the curve, and there does not seem to be an explicit form for a
block semidefinite representation. When the curve is non-rational,
the proof becomes technically even much more complicated.
\end{rem}


\section{Sparse non-negativity certificates in polynomial optimization} \label{sect:sparse}

\subsection{Sums of copositive fewnomials and sparse semidefinite relaxations}
We are going to take a second look at Corollary \ref{cor2moncasthm},
concentrating on part~(b). A univariate polynomial $f\in\R[t]$ will
be called a \emph{$k$-nomial} if $f$ is a linear combination of at
most $k$ monomials $t^i$. We say that $f$ is \emph{copositive} if
$f\ge0$ on $\R_\plus$. Given a finite set
$J \subset\Z_\plus :=\{0,1,2,\dots\}$, let
\[
	\CP(J) := \setcond{p \ \in \R[t]}{p \ge 0 \ \text{on} \ \R_+ \ \text{and}  \ \supp(p) \subseteq J}.
\]
This is the cone $P$
considered in Sect.~3, for the monomial curve corresponding to~$J$.
Let
\[
	\SOCF_{k,d} := \sum_{\substack { J \subseteq \{0,\ldots,d\} \, : \\ |J| = k}} \CP(J)
\]
The acronym $\SOCF$ stands for \emph{sums of copositive fewnomials}. So
$f\in\R[t]_d$ lies in $\SOCF_{k,d}$ if, and only if, $f$ can be
written as a finite sum of copositive $k$-nomials of degree at
most~$d$. It is obvious how to generalize this setup from univariate
polynomials to a multivariate setting. This could be an interesting
setup to explore in the future.
	
As before, let $\Sigma_{2k}\subset\R[t]$ denote the set of sums of
squares polynomials of degree $\le2k$. The main result of Sect.~3
implies the following positivstellensatz for the cones SOCF:
	
\begin{cor}\label{pos:stellen:satz}%
For all integers $k,\,d\ge1$ with $d>2k$, we have
\begin{equation}\label{socf:repr}%
\SOCF_{2k+1,d}\>=\>t^0 \, \Sigma_{2k}+\cdots+t^{d-2k} \,\Sigma_{2k}.
\end{equation}
Moreover $\sxdeg(\SOCF_{2k+1,d})=k+1$.
\end{cor}

For $d=2k$, note that $\SOCF_{2k+1,2k}$ is just the cone of all
copositive univariate polynomials of degree at most $2k$, i.e.,
it is the cone
\[
\setcond{ f \in \R[t]_{2k}}{f \ge 0 \ \text{on} \
\R_+} = \Sigma_{2k}+t\,\Sigma_{2k-2},
\]
where the latter description of copositivity in the univariate case
is well known.

\begin{proof}
The inclusion ``$\subset$'' holds since, for every set
$J \subset\{0,\dots,d\}$ with $|J|\le2k+1$, the cone $\CP(J)$ is
contained in the right hand side of \eqref{socf:repr} (Corollary
\ref{cor2moncasthm}(b)).
The reverse inclusion is obvious. The last
statement follows from identity \eqref{socf:repr} by using
\cite{sch:sxdeg} Lemma~1.4(d).
\end{proof}

To describe the duals of the SOCF cones, we identify
$v \in (\R[t]_d)^{\vee}$ and $(v_0,\ldots,v_d) \in \R^{d+1}$
via $v_i = v(t^i)$.
\begin{lem}\label{socfdual}%
Let $d,\,k\ge1$ with $d>2k$. The dual of $\SOCF_{2k+1,d}$ is the cone
$$\bigl\{v=(v_0,\dots,v_d)\in\R^{d+1}\colon M_{0,k}(v)\succeq0,\,
\dots,\,M_{d-2k,k}(v)\succeq0\bigr\},$$
where
\[
	M_{s,k}(v):=\bigl(v_{s+i+j}\bigr)_{i,j=0,\dots,k} \in \R^{(k+1) \times (k+1)}.
\]
\end{lem}

\begin{proof}
This follows from \eqref{socf:repr} in Corollary~\ref{pos:stellen:satz},
since $(C_1+C_2)^*=C_1^*\cap C_2^*$ and
since $(\Sigma_{2k})^*=\{u=(u_0,\dots,u_k)\colon M_{0,k}(u)\succeq0\}$;
see, for example, \cite{ble} Sect.~4.6.
\end{proof}

The notation $M_{s,k}(v)$ from Lemma~\ref{socfdual} is used in the
following proposition, which provides   primal and dual sparse
relaxations for the problem $\min_{\R_+} f$, with $f \in \R[t]_d$,
with a flexible choice of the sparsity threshold $2k+1$.

\begin{prop}\label{sparserelax}%
Let $k,d$ be positive integers with $2k <d$ and
 $f\in\R[t]_d$. Then the following
elements in $\R\cup\{\pm\infty\}$ coincide:
\begin{enumerate}
\item \label{socf:relax}
$\sup\{\lambda\in\R\colon f-\lambda\in\SOCF_{2k+1,d}\}$,
\item \label{sparse:sos}
$\sup\{\lambda\in\R\colon f-\lambda\in t^0 \, \Sigma_{2k}+\cdots+
t^{d-2k} \,\Sigma_{2k}\}$,
\item \label{sparse:moment}
$\inf\{\bil fv\colon v=(v_0,\dots,v_d)\in\R^{d+1}$, $v_0=1$,
$M_{0,k}(v),\dots,M_{d-2k,k}(v)\succeq0\}$.
\end{enumerate}
\end{prop}

\begin{proof}
This follows from \eqref{socf:repr} in Corollary
\ref{pos:stellen:satz}, using Proposition \ref{conic:opt:duality}.
\end{proof}

\begin{rem}
For the case $k=0$, relaxations in Proposition~\ref{sparserelax} are
trivial with the optimal value $f(0)$ or $-\infty$, depending on
whether or not all coefficients of $f- f(0)$ are non-negative. But
already the simplest non-trivial case $k=1$ exhibits connections to
active ongoing research on sparse relaxations in polynomial
optimization, namely to the relaxations  that rely on the so-called
\emph{sage} and \emph{sonc} polynomials; see \cite{dre} and \cite{mur,kat},
respectively.  Both kinds of polynomials emerge from the same idea
applied to non-negativity on  $\R_+^n$ and $\R^n$, respectively.
A sage polynomial is a  sum of  polynomials that are non-negative on
$\R_+^n$ and whose supports are contained in a simplicial curcuit.
Here,  a simplicial circuit is an inclusion-minimal affinely dependent
set whose convex hull is a simplex. Since simplicial circuits within
$\R$ are merely three-element sets, we see that the cone  $\SOCF_{3,d}$
is the cone of all univariate sage polynomials of degree at most $d$.
Thus, the case $k=1$ provides a primal and a dual sage relaxation of
$\min_{\R_+} f$. See, for example, \cite{kat} for the discussion of
the duality for the sonc and sage cones and the applications in
optimization.  We also note that the special case $k=1$ of
Corollary~\ref{pos:stellen:satz} amounts to a description of
non-negative sage polynomials in terms of the so-called reduced
circuits; see \cite{kat}. That is, in the univariate situation,
an analog of a reduced circuit in the setting of copositive
$(2k+1)$-nomials is a set of $2k+1$ consecutive integer values, since
the support of polynomials in $t^i \Sigma_{2k}$ is a subset of
$\{i,\ldots,i+2k\}$.
	
We  stress that the sparsity threshold of sage and sonc polynomials
is bound to the dimension, since a circuit in dimension $n$ cannot
have more than $n+2$ elements. In contrast to this, in our setting
the sparsity threshold $2k+1$ may vary arbitrarily, bridging the sage
relaxations with the dense standard relaxations of Lasserre. While
being aware that our results are limited to the univariate case, we
hope that such results might serve as an inspiration for similar
studies in an arbitrary dimension $n$. That is, it would be
interesting to understand properties of sos and moment  relaxations
that have a variable sparsity threshold.
\end{rem}

\begin{rem}
The above discussion of sparse polynomial optimization over $\R_+$
can also be used to handle sparse polynomial optimization on an
arbitrary interval $I \subseteq \R_+$. For example, when $2k < d -1$
and $I=[0,1]$, in view of Corollaries~\ref{cor2moncasthm}(a)  and
\ref{pos:stellen:satz}, the cone of sums of $(2k+1)$-nomials
of degree at most $d$  that are non-negative on $[0,1]$ can be
described as $\SOCF_{2k+1,d}  + (t-1) \SOCF_{2k,d-1}$.
For both of the involved SOCF cones  we have semidefinite
descriptions of  size $k+1$. This leads to a semidefinite formulation
of the optimization problem $\inf_{t \in [0,1]} p(t)$ for $p$ a
polynomial with $|\supp(p)| \le 2k+1$ using $\mathcal{O}(d)$ LMIs of
size $k+1$, analogous to the formulation \eqref{sparse:moment} given
in Proposition~\ref{sparserelax}.
\end{rem}

\subsection{Sparse semidefinite relaxations with a chordal-graph sparsity} \label{par:chordal:sparsity}	We consider graphs $G=(V,E)$ where $|V| < \infty$ and edges $e \in E$ are two-element subsets of $V$. A \emph{cycle} of  length $k$ in $G$ is a connected subgraph of $G$ with $k$ nodes and each node having degree two. A \emph{chord} of a cycle $C$ in $G$ is an edge of $G$ that connects two nodes of $C$ but is not an edge of $C$. A graph $G$ is said to be \emph{chordal} if every cycle of $G$ of length at least four has a chord in $G$. A \emph{clique} $W \subseteq V$ of $G=(V,E)$ is a set of nodes with any two distinct nodes in $W$ connected by an edge. The following result provides a convenient representation of the cone of psd matrices with the chordal sparsity pattern:

\begin{thm}[{see \cite{agler} and \cite{van} Sect.~9.2}] \label{chordal:psd:thm}
Let $G=(V,E)$ be a graph with $n = |V|<\infty$. With $G$ we associate the cone
\[
\sfS_{+,E}^V := \setcond{ (a_{ij})_{i,j \in V}
\in \sfS_+^n}{a_{ij} = 0 \ \text{when}\ i \ne j \ \text{and} \ \{i,j\} \not\in E}.
\]
If $G$ is chordal, then this cone admits the decomposition
\[
\sfS_{+,E}^V = \sfS_{+}^{V_1} + \cdots + \sfS_{+}^{V_N},
\]
where $V_1,\ldots,V_N \subseteq V$ are all inclusion-maximal cliques
of $G$ and, for $W \subseteq V$,
\[
\sfS_+^{W} := \setcond{ (a_{ij})_{i,j\in V} \in \sfS_+^n}
{ a_{ij} =0 \ \text{for} \ (i,j) \not \in W \times W}.
\]
\end{thm}

\begin{rem}
$\sfS_{+,E}^V$ is the so-called cone of psd matrices with the sparsity
pattern of $G=(V,E)$.
If $G$ is chordal, $\sfS_{+,E}^V$  is known to have favorable theoretical
and computational properties, which allow one to solve conic optimization
problems with respect to the cones $\sfS_{+,E}^V$ more efficiently than
the problems with respect to  $\sfS_+^n$.
For example,
computations with barrier functions can be done more efficiently on
$\sfS_{+,E}^V$ since there are more efficient ways to carry out Cholesky
factorization for matrices in $\sfS_{+,E}^V$ when the graph $G=(V,E)$ is
chordal. We refer to \cite{zheng} Appendix~A for further details.
\end{rem}

\begin{rem}
As a consequence of Theorem~\ref{chordal:psd:thm} we see that, for a
chordal graph $G=(V,E)$, the semidefinite extension degree of
$\sfS_{+,E}^V$ is the clique number of $G$, i.e., the maximum clique
size of $G$.
\end{rem}

For positive integers $n$ and $k$ with $k < n$, the cone
\begin{equation*}
\sfS_{+,k}^{n} = \setcond{ (a_{ij} )_{i,j=0,\ldots,n-1} \in
\sfS_+^n }{a_{ij} =0 \ \text{for} \ |i-j| > k}
\end{equation*}
is the cone $\sfS_{+,E}^V$ for the graph $G = (V,E) $ with vertex set
$V = \{0,\ldots,n-1\}$ and edge set $E$ consisting of  $\{i,j\}$ that
satisfy $0  < |i-j| \le k$. In \cite{van} Sect.~8.2, $\sfS_{+,k}^n$
is called the cone of psd matrices with the  \emph{band sparsity
pattern of the band width $2k+1$}.   It is clear that the graph $G$
defining $\sfS_{+,k}^n$ is chordal and that inclusion-maximal cliques
of  $G$ are sets of the form $\{i,\ldots,i+k\}$ with $0 \le i < n -k$.
We show that, similarly to how  $\sfS_+^n$  is used for  representing
sos cones and truncated quadratic modules, the cone  $\sfS_{+,k}^n$
can be used for representing the cone $\SOCF_{2k+1,n}$. That is,
optimization problem \eqref{socf:relax} from Proposition~\ref{sparserelax}
can be formulated as a conic problem with respect to the cones
$\sfS_{+,k}^n$. Such  formulations may have computational advantages
if used in solvers that can exploit the chordal sparsity.

\begin{prop}
Let $d,k$ be positive integers with $2 k  < d$.
Then the cone $\SOCF_{2k +1,d}$ admits a representation as a linear
image of  $\sfS_{+,k}^{e+1} \times \sfS_{+,k}^e$, if $d=2e$ is even,
and as a linear image of
$\sfS_{+,k}^{e+1} \times \sfS_{+,k}^{e+1},$ if $d = 2e+1$ is odd.
\end{prop}

\begin{proof}
We only consider the case of an even $d$ with  $d = 2e$, because the
case of an odd $d$ is completely analogous. We use the linear map
$\varphi : \sfS^{e+1} \to \R[t]_d$ with $\varphi(A) := \sum_{i,j=0}^e a_{ij} t^{i+j}$
for $A = (a_{ij})_{i,j=0,\ldots,e} \in \sfS^{e+1}$ from
Remark~\ref{semidefblockrep}. As mentioned in  Remark~\ref{semidefblockrep},
$\varphi(\sfS_+^W) = \Sigma_{2k}$ for $W = \{0,\ldots,k\}$, where the
notation $\sfS_+^W$ is borrowed from	Theorem~\ref{chordal:psd:thm}.
The latter implies $\varphi(\sfS_+^{V_i}) = t^{2i} \Sigma_{2k}$ for
$V_i := \{i,\ldots,i+k\}$ with $0 \le i \le d - 2k$. Consequently,
splitting the sum representing $\SOCF_{2k+1,d}$ in \eqref{socf:repr}
into two according to the parity of the exponents in $t^0,\ldots,t^{d-2k}$,
we obtain
\begin{align*}
\SOCF_{2k+1,d}
& = \sum_{i=0}^{e-k} t^{2i} \Sigma_{2k} + t \, \sum_{i=0}^{e-k-1} t^{2i} \Sigma_{2k}
\\ & = \phi \biggl(\underbrace{\sum_{i=0}^{e-k} \sfS_+^{V_i}}_{=:K} \biggr) + t \,
\phi \biggl(\underbrace{\sum_{i=0}^{e-k-1} \sfS_+^{V_i}}_{=:L} \biggr).
\end{align*}
Thus, $\SOCF_{2k+1,d}$ is a linear image of $K \times L$ under the
linear map $(A,B) \mapsto \varphi(A) + t \, \varphi(B)$. In view of
Theorem~\ref{chordal:psd:thm}, $K$ and $L$ are copies of $\sfS_{+,k}^{e+1}$
and $\sfS_{+,k}^e$ respectively. This gives the desired assertion.
\end{proof}


\begin{rem}
Wang, Magron and Lasserre \cite{wang1,wang} suggest to use the cones
$\sfS_{+,E}^V$ from chordal graphs $G=(V,E)$ to approximate sparse sos
polynomials. For a given finite set $A \subseteq \Z_+^n$, they consider
the cone
\[
\Sigma(A) := \setcond{ f \in \R[x_1,\ldots,x_n] }{f \ \text{sos and}
\ \supp(f) \subseteq A}
\]
They suggest two iterative algorithms (one in \cite{wang} and another
one in \cite{wang}) that take $A$ as an input and produce a chordal
graph $G = (V,E)$ with $V \subseteq \Z_+^n$ in order to use the image
of the repsective cone $\sfS_{+,E}^V$ under the map
\[
\varphi: (a_{\alpha,\beta})_{\alpha,\beta \in V} \mapsto
\sum_{\alpha,\beta \in V} a_{\alpha,\beta} x^{\alpha+\beta}
\]
as an approximation of $\Sigma(A)$. The algorithm in \cite{wang1}
produces a $G=(V,E)$ being a disjoint union of cliques. The graph
$G=(V,E)$ in \cite{wang1} is guaranteed to satisfy the equality
\begin{equation} \label{sigma:A:descr}
\Sigma(A) =  \setcond{f \in \phi(\sfS_{+,E}^V)}{\supp(f) \subseteq A}
\end{equation}
(see Thm.~3.3 in \cite{wang1}), but there is no guarantee that the
clique number of $G$ is small, or rather, the dependence of the clique
number of $G$ on the properties of $A$ remains unexplored.  Since the
graph $G=(V,E)$ generated by an algorithm from \cite{wang1} may have
large cliques, in \cite{wang} another approach is suggested that
generates a graph $G=(V,E)$ with a smaller number of edges. This
other approach is heuristic in the sense that there are no theoretical
guarantees for the equality $\Sigma(A) = \phi(\sfS_{+,E}^V)$ (see
Example~3.5. in \cite{wang}). It would be interesting to study the
semidefinite extension degree of the cones $\Sigma(A)$ and try to
relate these cones to the cones $\sfS_{+,E}^V$ in a non-heuristic way.
\end{rem}


\end{document}